\pdfoutput=1 
\documentclass[draft=false, leqno]{article}
\usepackage{amsmath,amsthm,mathrsfs}
\usepackage[normal]{boilerart1}
\usepackage{longtable}

\newcommand{\textbfit}[1]{\textbf{\textit{#1}}}

\theoremstyle{definition}
\newtheorem{rmk}[equation]{Remark}
\newtheorem*{ack*}{Acknowledgments}

\hypersetup{
    urlcolor = slate,
}

\newcommand{\andeq}{\text{\qquad and\qquad}}
\newcommand{\period}{\rlap{\ .}}
\newcommand{\comma}{\rlap{\ ,}}
\newcommand{\defn}[1]{\textbfit{#1}}

\newcommand{\flowerstar}{f_{\ast}}
\newcommand{\fupperstar}{f^{\ast}}

\newcommand{\pupperstar}{p^{\ast}}

\newcommand{\Nis}{\ensuremath{\textit{nis}}}

\newcommand{\proet}{\ensuremath{\textit{proét}}}

\newcommand{\zar}{\ensuremath{\textit{zar}}}

\newcommand{\bdd}{\ensuremath{\textit{b}}}

\newcommand{\qc}{\ensuremath{\textit{qc}}}

\newcommand{\coh}{\ensuremath{\textit{coh}}}
\newcommand{\bc}{\ensuremath{\textit{bc}}}

\newcommand{\pre}{\ensuremath{\textit{pre}}}

\newcommand{\cross}{\times}

\newcommand{\equivalent}{\simeq}

\newcommand{\paren}[1]{\left(#1\right)}

\renewcommand{\rhd}{\smalltriangleright}

\def\reflectedop#1#2{\mathop{\reflectbox{$#1{#2}$}}}

\newcommand{\Sh}{\categ{Sh}}
\newcommand{\Sheff}[1]{\Sh_{\eff}(#1)}

\newcommand{\preTopord}{\pre\Top}
\newcommand{\preTop}{\pre\Top_{\infty}}
\newcommand{\preTopbdd}{\preTop^{\bdd}}
\newcommand{\Topcohord}{\Top^{\coh}}
\newcommand{\Topcoh}{\Top_{\infty}^{\coh}}
\newcommand{\Topcohbdd}{\Top_{\infty}^{\bc}}
\newcommand{\Topbc}{\Topcohbdd}

\renewcommand{\Space}{\categ{Spc}}

\newcommand{\XXcoh}{\XX^{\coh}}
\newcommand{\YYcoh}{\YY^{\coh}}

\newcommand{\XXcohbdd}{\XXcoh_{<\infty}}

\usepackage{todonotes}

\title{On coherent topoi \& coherent $ 1 $-localic $ \infty $-topoi}

\author{Peter J. Haine}

\date{\today}

\begin{document}

\maketitle

\begin{abstract} 
	In this note we prove the following useful fact that seems to be missing from the literature: the $ \infty $-category of coherent ordinary topoi is equivalent to the $ \infty $-category of coherent $ 1 $-localic $ \infty $-topoi.
	We also collect a number of examples of coherent geometric morphisms between $ \infty $-topoi coming from algebraic geometry. 
\end{abstract}

\setcounter{tocdepth}{2}
\tableofcontents


\section*{Overview}\addcontentsline{toc}{section}{Overview}

Let $ f \colon \fromto{X}{Y} $ be a morphism between quasicompact quasiseparated schemes.
It follows from \cite[Example 7.1.7]{Ultracategories} that the induced geometric morphism 
\begin{equation*}
	\flowerstar \colon \fromto{\Sh_{\proet}(X;\Set)}{\Sh_{\proet}(Y;\Set)}
\end{equation*}
on proétale topoi is a coherent geometric morphism between coherent topoi in the sense of \cite[Exposé VI]{MR50:7131}.
It is often helpful to be able to apply methods of homotopy theory to topos theory, especially if one needs to work with stacks.
To do this, one works with the $ 1 $-localic $ \infty $-topos associated to an ordinary topos, obtained by taking sheaves of \textit{spaces} rather than sheaves of sets.
There is again an induced geometric morphism 
\begin{equation*}
	\flowerstar \colon \fromto{\Sh_{\proet}(X;\Space)}{\Sh_{\proet}(Y;\Space)} \comma
\end{equation*}
and these $ \infty $-topoi are coherent in the sense of \cite[\SAGapp{A}]{SAG}.
One naturally expects this geometric morphism to satisfy the same kinds of good finiteness conditions as the morphism of ordinary topoi does, i.e., be \textit{coherent} in the sense of \cite[\SAGapp{A}]{SAG}.
However, a proof of this fact is not currently in the literature.
This claim is not completely obvious either: from the perspective of higher topos theory, the pullback in a coherent geometric morphisms of ordinary topoi is only required to preserve $ 0 $-truncated coherent objects, rather than \textit{all} coherent objects.

In this note we fill this small gap in the literature.
We show that the theories of coherent ordinary topoi and coherent geometric morphisms (in the sense of \cite[Exposé VI]{MR50:7131}) and of coherent $ 1 $-localic $ \infty $-topoi and coherent geometric morphisms (in the sense of \cite[\SAGapp{A}]{SAG}) are equivalent (\Cref{prop:coherent1localic}).
This point is surely known to experts, but does not seem to be explicitly addressed in \cite[\SAGapp{A}]{SAG} or elsewhere.
Our main aim in proving this equivalence is to make the $ \infty $-categorical version of sheaf theory more accessible to (non-derived) algebraic geometers who are interested in applying results from \cite[\SAGapp{A}]{SAG} to ordinary coherent topoi.

The proof of this equivalence reduces to showing that a coherent geometric morphism of ordinary coherent topoi induces a coherent geometric morphism of corresponding $ 1 $-localic $ \infty $-topoi.
This follows from the more general fact that a morphism of finitary $ \infty $-sites induces a coherent gometric morphism on corresponding $ \infty $-topoi (\Cref{cor:morsitescoherent}).
In ordinary topos theory this is well-known \cite[Exposé VI, Corollaire 3.3]{MR50:7131}, but the $ \infty $-toposic version seems to be missing from the literature.

Our original motivation for proving \Cref{prop:coherent1localic} was the following.
In recent work with Barwick and Glasman \cite{exodromy} we proved a basechange theorem for oriented fiber product squares of bounded coherent $ \infty $-topoi \cite[Theorem 8.1.4]{exodromy}. 
In the original version of \cite{exodromy}, we claimed \cite[Corollary 8.1.6]{exodromy} that this implies the basechange theorem for oriented fiber products of coherent topoi of Moerdijk and Vermeulen \cite[Theorem 2(i)]{MR1731050} (which is the nonabelian refinement of a result of Gabber \cite[Exposé XI, Théorème 2.4]{MR3309086}).
While this is true, our original proof implicitly used that a coherent geometric morphism of ordinary topoi induces a coherent geometric morphism on corresponding $ 1 $-localic $ \infty $-topoi.
 
In \cref{sec:review} we review the classification of coherent topoi in terms of pretopoi as well as the classification of bounded coherent $ \infty $-topoi in terms of bounded $ \infty $-pretopoi.
This review is aimed at readers familiar with \cite[Exposé VI]{MR50:7131}, but not necessarily with pretopoi or coherent $ \infty $-topoi; the familiar reader should skip straight to \cref{sec:cohfor1localic}.
At the end of \cref{sec:cohfor1localic} we collect a number of examples of coherent geometric morphisms between $ \infty $-topoi coming from algebraic geometry.

\begin{ack*}
	We thank Clark Barwick for his guidance and sharing his many insights about this material.
	We also gratefully acknowledge support from both the \textsc{mit} Dean of Science Fellowship and \textsc{nsf} Graduate Research Fellowship.
\end{ack*}


\subsection{Terminology \& notations}

\begin{itemize}
	\item We write $ \NN $ for the poset of \textit{nonnegative integers}, and $ \NN^{\rhd} \coloneq \NN \cup \{\infty \} $.

	\item We write $ \Cat_{\infty} $ for the $ \infty $-category of $ \infty $-categories.

	\item We write $ \Top_{\infty} \subset \Cat_{\infty} $ for the $ \infty $-category of $ \infty $-topoi and geometric morphisms.
	We typically write $ \flowerstar \colon \fromto{\XX}{\YY} $ to denote a geometric morphism from an $ \infty $-topos $ \XX $ to an $ \infty $-topos $ \YY $ and write $ \fupperstar $ for the left exact left adjoint of $ \flowerstar $.

	\item We write $ \Cat $ for the $ (2,1) $-category of (ordinary) categories, functors, and natural isomorphisms, which we tacitly regard as an $ \infty $-category (via the Duskin nerve \kerodon{009P}).
	We write $ \Top \subset \Cat $ for the subcategory of topoi and geometric morphisms. 
\end{itemize}


\section{Premilinaries on (higher) coherent topoi \& pretopoi}\label{sec:review}

In this section we review the classification of coherent topoi in terms of pretopoi, as well as the theory of coherent $ \infty $-topoi and the classification of \textit{bounded} coherent $ \infty $-topoi in terms of \textit{bounded} $ \infty $-pretopoi.


\subsection{Classification of coherent topoi}\label{subsec:coh1topoi}

We assume that the reader is familiar with coherent topoi in the sense of \cite[Exposé VI]{MR50:7131}.
Excellent accounts of coherent topoi can also be found in \cites{Lurie:CatLogic11}[\S\S C.5 \& C.6]{Ultracategories}.
The classification of coherent topoi in terms of pretopoi is sketched in \cite[Exposé VI, Exercise 3.11]{MR50:7131}; a self-contained account can be found in \cite{Lurie:CatLogic13}.

\begin{dfn}\label{def:qcqsin1topos}
	Let $ \XX $ be a topos.
	\begin{enumerate}[(\ref*{def:qcqsin1topos}.1)]
		\item An object $ U \in \XX $ is \defn{quasicompact} if every covering of $ U $ has a finite subcovering.

		\item An object $ U \in \XX $ is \defn{quasiseparated} if for every pair of morphisms $ \fromto{U'}{U} $ and $ \fromto{U''}{U} $ where $ U' $ and $ U'' $ are quasicompact, the fiber product $ U' \cross_U U'' $ is quasicompact.

		\item An object $ U \in \XX $ is \defn{coherent} if $ U $ is quasicompact and quasiseparated.

		\item The topos $ \XX $ is \defn{coherent} if the terminal object $ 1_{\XX} \in \XX $ is coherent, every object of $ \XX $ admits a cover by coherent objects, and the coherent objects of $ \XX $ are closed under finite products.
	\end{enumerate}
	We write $ \XXcoh \subset \XX $ for the full subcategory spanned by the coherent objects.

	A geometric morphism of topoi $ \flowerstar \colon\fromto{\XX}{\YY} $ is \defn{coherent} if and only if, for every coherent object $ F \in \YY $, the object $ \fupperstar(F) \in \XX $ is coherent. 
	We write $\Topcohord $ for the subcategory of $\Top$ whose objects are coherent topoi and whose morphisms are coherent geometric morphisms.
\end{dfn}




\begin{dfn}[{\cite[Definition A.4.1]{Ultracategories}}]\label{def:1pretopos}
	A category $ X $ is a \defn{pretopos} if $ X $ satisfies the following conditions:
	\begin{enumerate}[(\ref*{def:1pretopos}.1)]
		\item The category $ X $ admits finite limits.

		\item The category $ X $ admits finite coproducts, which are universal and disjoint.

		\item Equivalence relations in $ X $ are effective.

		\item Effective epimorphisms in $ X $ are stable under pullback
	\end{enumerate}
	
	If $ X $ and $ Y $ are pretopoi, we say that a functor $ \fupperstar \colon \fromto{Y}{X} $ is a \defn{morphism of pretopoi} if $ \fupperstar $ preserves finite limits, finite coproducts, and effective epimorphisms.
	Write $ \preTopord \subset \Cat $ for the subcategory consisting of \textit{essentially small} pretopoi and morphisms of pretopoi.
\end{dfn}

\begin{exm}[{\cite[Corollary C.5.14]{Ultracategories}}]
	Let $ \XX $ be a coherent topos.
	Then the full subcategory $ \XXcoh \subset \XX $ of coherent objects is an essentially small pretopos.
	If $ \flowerstar \colon \fromto{\XX}{\YY} $ is a coherent geometric morphism of coherent topoi, then the functor $ \fupperstar \colon \fromto{\YYcoh}{\XXcoh} $ is a morphism of pretopoi.

	If $ \XX $ is the étale topos of a quasicompact quasiseparated scheme $ X $, then $ \XX $ is coherent and $ \XXcoh $ is the category of constructible étale sheaves of sets on $ X $.
\end{exm}

\begin{dfn}[{\cite[Definition B.5.3]{Ultracategories}}]
	Let $ X $ be a pretopos.
	The \defn{effective epimorphism topology} on $ X $ is the Grothendieck topology $ \eff $ on $ X $ where a collection of morphisms $ \{U_i \to U\}_{i \in I} $ is a covering if and only if there exists a finite subset $ I_0 \subset I $ such that the induced morphism $ \fromto{\coprod_{i \in I_0} U_i}{U} $ is an effective epimorphism in $ X $.

	The effective epimorphism topology is subcanonical \cite[Corollary B.5.6]{Ultracategories}.
\end{dfn}



\begin{thm}[{\cites[Corollary 7]{Lurie:CatLogic13}[Proposition C.6.3]{Ultracategories}}]\label{thm:classcoherent1topoi}
	The constructions $\goesto{\XX}{\XXcoh}$ and $\goesto{X}{\Sheff{X;\Set}}$ are mutually inverse equivalences of $ (2,1) $-categories
	\begin{equation*}
		\Topcohord \equivalent \preTopord^{\op} \period
	\end{equation*}
\end{thm}

\begin{rmk}
	The equivalence of \Cref{thm:classcoherent1topoi} is really an equivalence of $ (2,2) $-categories, but we do not need noninvertible $ 2 $-morphisms in this note.
\end{rmk}


\subsection{Classification of bounded coherent \texorpdfstring{$\infty$}{∞}-topoi}\label{subsec:cohtopoi}

Coherent $ \infty $-topoi admit a classification in terms of a higher-categorical analogue of pretopoi, as long as they can be recovered from the collection of their $ n $-topoi of $ (n-1) $-truncated objects.
This subsection is a breif summary of \cite[\S\S \SAGseclink{A.2}, \SAGseclink{A.3}, \SAGseclink{A.6}, \& \SAGseclink{A.7}]{SAG}.

\begin{ntn}
	We use here the theory of \textit{$ n $-topoi} for $n\in\NN^{\rhd}$; see \cite[\HTTch{6}]{HTT}.
	We write $\Top_n\subset\Cat_{\infty}$ for the subcategory of $ n $-topoi and geometric morphisms.
\end{ntn}

\begin{exm}
	Recall that $1$-topoi are topoi in the classical sense \HTT{Remark}{6.4.1.3}. 
\end{exm}

\begin{exm}
	Let $ m,n \in\NN^{\rhd} $ with $ m\leq n $.
	An \defn{$ m $-site} is a small $ m $-category\footnote{By an \defn{$ m $-category} we mean an $ \infty $-category whose mapping spaces are $ (m-1) $-truncated.} $ X $ equipped with a Grothendieck topology $ \tau $.
	Attached to this $ m $-site is the $ n $-topos $\Sh_{\tau,\leq(n-1)}(X)$ of sheaves of $ (n-1) $-truncated spaces on $ X $.
	We simply write $ \Sh_{\tau}(X) $ for the $ \infty $-topos of sheaves of spaces on $ X $.

	Not all $ \infty $-topoi are of the form $\Sh_{\tau}(X)$ for some $ \infty $-site $X$; however, if $n\in\NN$, then every $ n $-topos is of the form $\Sh_{\tau,\leq(n-1)}(X)$ for some $ n $-site $(X,\tau)$ \cite[\href{http://www.math.harvard.edu/~lurie/papers/HTT.pdf\#theorem.6.4.1.5}{Theorem 6.4.1.5(1)}]{HTT}.
\end{exm}

 
\begin{dfn}[{\cite[\HTTsubsec{6.4.5}]{HTT}}]\label{cnstr:localictopoi}
	For any integer $ n \geq 0 $, passage to $ (n-1) $-truncated objects defines a functor $\tau_{\leq n-1}\colon\fromto{\Top_{\infty}}{\Top_n} $.
	The functor $ \tau_{\leq  n-1} $ admits admits a fully faithful right adjoint $ \incto{\Top_n}{\Top_{\infty}} $ whose essential image we denote by $ \Top_{\infty}^n \subset \Top_{\infty} $.
	The $ \infty $-category $ \Top_{\infty}^n $ is the $ \infty $-category of \defn{$ n $-localic} $ \infty $-topoi.
\end{dfn}

\begin{exm}
	For any topological space $ T $, the $ \infty $-topos $ \Sh(T) $ of sheaves on $ T $ is $ 0 $-localic.
\end{exm}

\begin{exm}
	If $ \XX $ is a topos presented as sheaves of sets on a site $ (X,\tau) $ \textit{with finite limits}, then the $ 1 $-localic $ \infty $-topos associated to $ \XX $ is the $ \infty $-topos $ \Sh_{\tau}(X) $ of sheaves of \textit{spaces} on $ (X,\tau) $.
\end{exm}

\begin{nul}
	Let $ n \in \NN $.
	The proof of \HTT{Proposition}{6.4.5.9} demonstrates that an $ \infty $-topos $ \XX $ is $ n $-localic if and only if $\XX \equivalent \Sh_{\tau}(X)$ for some $ n $-site $ (X,\tau) $ \textit{with finite limits}.
\end{nul}

\begin{wrn}
	If $ (X,\tau) $ is an $ n $-site and the $ n $-category $ X $ does \textit{not} have finite limits, then the $ \infty $-topos $ \Sh_{\tau}(X) $ is not generally $ N $-localic for any integer $ N \geq 0 $.
	See \SAG{Counterexample}{20.4.0.1} for a basis $ B $ for the topology on the Hilbert cube $ \prod_{i \in \ZZ} [0,1] $ for which the $ \infty $-topos of sheaves on $ B $ is not $ N $-localic for any $ N \geq 0 $.
\end{wrn}


\begin{dfn}[\SAG{Definition}{A.7.1.2}]\label{def:boundedtopoi}
	An $ \infty $-topos $ \XX $ is \defn{bounded} if $ \XX $ can be written as the limit of a diagram $ \YY \colon \fromto{I}{\Top_{\infty}} $ where $ I^{\op} $ is a filtered $ \infty $-category and for each $ i \in I $ the $ \infty $-topos $ \YY_i $ is $ n_i $ localic for some $ n_i \in \NN $.
\end{dfn}

\begin{dfn}[\SAG{Definition}{A.2.0.12}]\label{def:coherence}
	Let $ \XX $ be an $ \infty $-topos.
	We say that $ \XX $ is \defn{$ 0 $-coherent} or \defn{quasicompact} if and only if every cover $ \{ U_i \to 1_{\XX}\}_{i \in I} $ of the terminal object $ 1_{\XX} \in \XX $ admits a finite subcover.
	Let $ n \geq 1 $ be an integer, and define $ n $-coherence of $ \infty $-topoi and their objects recursively as follows:
	\begin{enumerate}[(\ref*{def:coherence}.1)]
		\item An object $U\in \XX$ is \defn{$ n $-coherent} if and only if the $ \infty $-topos $\XX_{/U}$ is $ n $-coherent.

		\item The $ \infty $-topos $ \XX $ is \defn{locally $ n $-coherent} if and only if every object $U\in \XX$ admits a cover $ \{\fromto{U_i}{U}\}_{i\in I} $ where each $ U_i $ is $ n $-coherent.
		
		\item The $ \infty $-topos $ \XX $ is \defn{$(n+1)$-coherent} if and only if $ \XX $ is locally $ n $-coherent, and the $ n $-coherent objects of $ \XX $ are closed under finite products.
	\end{enumerate}


	An $ \infty $-topos $ \XX $ is \defn{coherent} if and only if $ \XX $ is $ n $-coherent for every $ n \geq 0 $.
	An object $U$ of an $ \infty $-topos $ \XX $ is \defn{coherent} if and only if $\XX_{/U}$ is a coherent $ \infty $-topos.
	Finally, an $ \infty $-topos $ \XX $ is \defn{locally coherent} if and only if every object $U\in \XX$ admits a cover $\{\fromto{U_i}{U}\}_{i\in I}$ where each $ U_i $ is coherent.
\end{dfn}

\begin{dfn}\label{def:cohgeommor}
	A geometric morphism of $ \infty $-topoi $ \flowerstar \colon\fromto{\XX}{\YY} $ is \defn{coherent} if and only if, for every coherent object $ F \in \YY $, the object $ \fupperstar(F) \in \XX $ is coherent. 
	We write $\Topcoh$ for the subcategory of $\Top_{\infty}$ whose objects are coherent $ \infty $-topoi and whose morphisms are coherent geometric morphisms.

	Write $\Topbc \subset \Topcoh $ for the full subcategory spanned by those coherent $ \infty $-topoi that are also bounded, that is, the \defn{bounded coherent} $ \infty $-topoi
\end{dfn}


\begin{ntn} 
	If $ \XX $ is an $ \infty $-topos, then write $ \XXcoh \subset \XX $ for the full subcategory of $ \XX $ spanned by the coherent objects and $ \XXcohbdd \subset \XX $ for the full subcategory of $ \XX $ spanned by the truncated coherent objects.
\end{ntn}

\begin{exm}
	The $ \infty $-topos $ \Space $ of spaces is coherent. 
	An object $ U \in \Space $ is truncated coherent if and only if $ U $ is a \textit{$ \pi $-finite} space, i.e., $ U $ is truncated, has finitely many connected components, and all of the homotopy groups of $ U $ are finite.
\end{exm}

\begin{dfn}[\SAG{Definition}{A.3.1.1}]\label{dfn:finitaryinftysite}
	An $ \infty $-site $(X,\tau)$ is \defn{finitary} if and only if $ X $ admits all fiber products, and, for every object $ U \in X $ and every covering sieve $ S \subset X_{/U} $, there is a finite subset $\{U_i\}_{i\in I} \subset S $ that generates a covering sieve.

	Let $ (X,\tau_X) $ and $ (Y,\tau_Y) $ be finitary $ \infty $-sites.
	A morphism of $ \infty $-sites $ \fupperstar \colon \fromto{(Y,\tau_Y)}{(X,\tau_X)} $ is a \defn{morphism of finitary $ \infty $-sites} if $ \fupperstar $ is preserves fiber products. 
\end{dfn}

\begin{prp}[\SAG{Proposition}{A.3.1.3}]\label{prop:SAG.A.3.1.3}
	Let $(X,\tau)$ be a finitary $ \infty $-site.
	Then the $ \infty $-topos $\Sh_{\tau}(X)$ locally coherent, and for every object $x\in X$, the sheaf $ \yo(x)$ is a coherent object of $\Sh_{\tau}(X)$, where $ \yo \colon \fromto{X}{\Sh_{\tau}(X)} $ is the sheafified Yoneda embedding.
	If, in addition, $ X $ admits a terminal object, then $ \Sh_{\tau}(X) $ is coherent.
\end{prp}

\begin{dfn}[\SAG{Definition}{A.6.1.1}]\label{def:pretopos}
	An $ \infty $-category $ X $ is an \defn{$ \infty $-pretopos} if $ X $ satisfies the following conditions:
	\begin{enumerate}[(\ref*{def:pretopos}.1)]
		\item The category $ X $ admits finite limits.

		\item The category $ X $ admits finite coproducts, which are universal and disjoint.

		\item Groupoid objects in $ X $ are effective, and their geometric realizations are universal.
	\end{enumerate}

	If $ X $ and $ Y $ are $ \infty $-pretopoi, we say that a functor $ \fupperstar \colon \fromto{Y}{X} $ is a \defn{morphism of $ \infty $-pretopoi} if $ \fupperstar $ preserves finite limits, finite coproducts, and effective epimorphisms.
	We write $ \preTop \subset \Cat_{\infty} $ for the subcategory consisting of $\infty $-pretopoi and morphisms of $ \infty $-pretopoi.
\end{dfn}

\begin{exm}[\SAG{Corollary}{A.6.1.7}]
	If $ \XX $ is a coherent $\infty $-topos, then the full subcategory $\XXcoh\subset\XX$ spanned by the coherent objects is an $\infty $-pretopos.
\end{exm}

\begin{dfn}[\SAG{Definition}{A.6.2.4}]\label{def:effepi}
	Let $ X $ be an $ \infty $-pretopos.
	The \defn{effective epimorphism topology} on $ X $ is the Grothendieck topology $ \eff $ where a collection of morphisms $ \{U_i \to U\}_{i \in I} $ is a covering if and only if there exists a finite subset $ I_0 \subset I $ such that the induced morphism $ \fromto{\coprod_{i \in I_0} U_i}{U} $ is an effective epimorphism in $ X $.
	
	The effective epimorphism topology is finitary and subcanonical \SAG{Corollary}{A.6.2.6}.
\end{dfn}

\begin{dfn}[\SAG{Definition}{A.7.4.1}]\label{dfn:boundedpretopos}
	An $\infty $-pretopos $ X $ is \defn{bounded} if and only if $ X $ is essentially small and every object of $ X $ is truncated.
	We write $\preTopbdd \subset \preTop$ for the full subcategory spanned by the bounded $\infty $-pretopoi.
\end{dfn}


\begin{thm}[\SAG{Theorem}{A.7.5.3}]\label{thm:SAG.A.7.5.3}
	The constructions $\goesto{\XX}{\XXcohbdd}$ and $\goesto{X}{\Sheff{X}}$ are mutually inverse equivalences of $\infty $-categories
	\begin{equation*}
		\Topbc \simeq \preTop^{\bdd,\op} \period
	\end{equation*}
\end{thm}


\section{Coherence for \texorpdfstring{$1$}{1}-localic \texorpdfstring{$\infty$}{∞}-topoi}\label{sec:cohfor1localic}

In this section we show that the $ \infty $-category of coherent ordinary topoi is equivalent to the $ \infty $-category of coherent $ 1 $-localic $ \infty $-topoi (\Cref{prop:coherent1localic}).
This follows from the fact that morphisms of finitary $ \infty $-sites induce coherent geometric morphisms (\Cref{cor:morsitescoherent}).
First we'll have to give $ \infty $-toposic versions of a number of points from \cite[Exposé VI, \S\S 1--3]{MR50:7131}, which follow easily from \cite[\SAGsubsec{A.2.1}]{SAG}.

\begin{dfn}\label{def:relativecoh}
	Let $ n \in \NN $ and let $ \XX $ be a locally $ n $-coherent $ \infty $-topos. 
	A morphism $ \fromto{U}{V} $ in $ \XX $ is \defn{relatively $ n $-coherent} if for every $ n $-coherent object $ V' \in \XX $ and every morphism $ \fromto{V'}{V} $, the fiber product $ U \cross_V V' $ is also $ n $-coherent.
\end{dfn}

\begin{exm}[\SAG{Example}{A.2.1.2}]\label{exm:SAG.A.2.1.2}
	Let $ \XX $ be a locally $ n $-coherent $ \infty $-topos and $ f \colon \fromto{U}{V} $ a morphism in $ \XX $.
	If $ U $ is $ n $-coherent and $ V $ is $ (n+1) $-coherent, then $ f $ is relatively $ n $-coherent.
\end{exm}
	
\begin{lem}\label{lem:quotientofqcisqc}
	Let $ \XX $ be an $ \infty $-topos.
	If $ e \colon \surjto{U}{V} $ is an effective epimorphism in $ \XX $ and $ U $ is quasicompact, then $ V $ is quasicompact.
\end{lem}

\begin{proof}
	This is a special case of \SAG{Proposition}{A.2.1.3}.
\end{proof}

\begin{lem}\label{lem:relativecoh}
	Let $ n \geq 1 $ be an integer and $ \XX $ a locally $ (n-1) $-coherent $ \infty $-topos.
	Let $ U \in \XX $ and let $ e \colon \surjto{\coprod_{i \in I} U_i}{U} $ be a cover of $ U $ where $ I $ is finite and $ U_i $ is $ n $-coherent for each $ i \in I $.
	The following are equivalent: 
	\begin{enumerate}[{\upshape (\ref*{lem:relativecoh}.1)}]
		\item The effective epimorphism $ e $ is relatively $ (n-1) $-coherent.

		\item For all $ i,j \in I $, the object $ U_i \cross_{U} U_j $ is $ (n-1) $-coherent.

		\item The object $ U $ is $ n $-coherent.
	\end{enumerate}
\end{lem}

\begin{proof}
	If $ e $ is relatively $ (n-1) $-coherent, then since coproducts in $ \XX $ are universal, the fiber product
	\begin{equation*}
		\paren{\textstyle \coprod_{i \in I} U_i} \cross_U \paren{\textstyle \coprod_{j \in I} U_j} \equivalent \coprod_{i,j \in I} U_i \cross_{U} U_j
	\end{equation*}
	is $ (n-1) $-coherent.
	Thus $ U_i \cross_{U} U_{j} $ is $ (n-1) $-coherent for all $ i,j \in I $ \SAG{Remark}{A.2.0.16}.

	If each $ U_i \cross_{U} U_j $ is $ (n-1) $-coherent, then since each $ U_i $ is $ n $-coherent the pullback of $ e $ along itself
	\begin{equation*}
		\surjto{\coprod_{i,j\in I} U_i \cross_{U} U_j}{\coprod_{i \in I} U_i}
	\end{equation*}
	is relatively $ (n-1) $-coherent (\Cref{exm:SAG.A.2.1.2}).
	Applying \SAG{Corollary}{A.2.1.5} we deduce that $ e \colon \surjto{\coprod_{i \in I} U_i}{U} $ is relatively $ (n-1) $-coherent.

	To conclude, note that if $ e \colon \surjto{\coprod_{i \in I} U_i}{U} $ is relatively $ (n-1) $-coherent, then \SAG{Proposition}{A.2.1.3} shows that $ U $ is $ n $-coherent.
	On the other hand, if $ U $ is $ n $-coherent, then $ e $ is $ (n-1) $-coherent by \Cref{exm:SAG.A.2.1.2}.
\end{proof}

\begin{prp}\label{prop:gencoherenceofgeom}
	Let $ \flowerstar \colon \fromto{\XX}{\YY} $ be a geometric morphism of $ \infty $-topoi and $ n \in \NN $.
	Assume that:
	\begin{enumerate}[{\upshape (\ref*{prop:gencoherenceofgeom}.1)}]
		\item There exists a collection of $ n $-coherent objects $ \YY_0 \subset \Obj(\YY) $ of $ \YY $ such that for every $ n $-coherent object $ U \in \YY $ there exists a cover $ \surjto{\coprod_{i \in I} U_i}{U} $ where $ U_i \in \YY_0 $ for each $ i \in I $.

		\item The pullback functor $ \fupperstar \colon \fromto{\YY}{\XX} $ takes objects of $ \YY_0 $ to $ n $-coherent objects of $ \XX $.

		\item If $ n \geq 1 $, the $ \infty $-topoi $ \XX $ and $ \YY $ are locally $ (n-1) $-coherent and $ \fupperstar \colon \fromto{\YY}{\XX} $ takes $ (n-1) $-coherent objects of $ \YY $ to $ (n-1) $-coherent objects of $ \XX $.
	\end{enumerate}
	Then $ \fupperstar $ takes $ n $-coherent objects of $ \YY $ to $ n $-coherent objects of $ \XX $.
\end{prp}

\begin{proof}
	Let $ U \in \YY $ be an $ n $-coherent object; we show that $ \fupperstar(U) $ is $ n $-coherent.
	By assumption there exists a cover
	\begin{equation*}
		e \colon \surjto{\coprod_{i \in I} U_i}{U}
	\end{equation*}
	where $ U_i \in \YY_0 $ for each $ i \in I $ and $ I $ is finite (since $ U $ is, in particular, $ 0 $-coherent).
	For all $ i \in I $ the object $ \fupperstar(U_i) $ is $ n $-coherent by assumption, so since $ n $-coherent objects are closed under finite coproducts \SAG{Remark}{A.2.0.16}, the object
	\begin{equation*}
		\fupperstar\paren{\textstyle \coprod_{i\in I} U_i} \equivalent \coprod_{i \in I} \fupperstar(U_i)
	\end{equation*}
	is $ n $-coherent.

	Note that 
	\begin{equation*}
		\fupperstar(e) \colon \surjto{\coprod_{i \in I} \fupperstar(U_i)}{\fupperstar(U)}
	\end{equation*}
	is an effective epimorphism in $ \XX $.
	If $ n = 0 $, this proves the claim (\Cref{lem:quotientofqcisqc}).
	If $ n \geq 1 $, then \Cref{lem:relativecoh} shows that it suffices to show that for all $ i ,j \in I $, the object 
	\begin{equation*}
		\fupperstar(U_i) \cross_{\fupperstar(U)} \fupperstar(U_j) \equivalent \fupperstar(U_i \cross_{U} U_j)
	\end{equation*}
	is $ (n-1) $-coherent.
	This follows from the fact that $ U_i \cross_{U} U_j $ is $ (n-1) $-coherent (by \Cref{lem:relativecoh}) and the assumption that $ \fupperstar $ sends $ (n-1) $-coherent objects of $ \YY $ to $ (n-1) $-coherent objects of $ \XX $.
\end{proof}

\Cref{prop:gencoherenceofgeom} shows that coherence of a geometric morphism  between locally coherent $ \infty $-topoi (\Cref{def:cohgeommor}) is equivalent to the \textit{a priori} stronger condition that the pullback functor preserve $ n $-coherent objects for all $ n \geq 0 $:\footnote{This second notion is how Grothendieck and Verdier originally defined coherence for ordinary topoi \cite[Exposé VI, Définition 3.1]{MR50:7131}.}

\begin{cor}
	Let $ \flowerstar \colon \fromto{\XX}{\YY} $ be a geometric morphism between locally coherent $ \infty $-topoi.
	Then $ \flowerstar $ is coherent if and only if $ \fupperstar $ takes $ n $-coherent objects of $ \YY $ to $ n $-coherent objects of $ \XX $ for all $ n \geq 0 $.
\end{cor}

\Cref{prop:gencoherenceofgeom} also shows that coherence of a geometric morphism can be checked on a generating set of coherent objects.

\begin{cor}\label{cor:checkcohongen}
	Let $ \flowerstar \colon \fromto{\XX}{\YY} $ be a geometric morphism between locally coherent $ \infty $-topoi.
	Let $ \YY_0 \subset \Obj(\YYcoh) $ be a collection of coherent objects such that for every object $ U \in \YY $ there exists a cover $ \surjto{\coprod_{i \in I} U_i}{U} $ where $ U_i \in \YY_0 $ for each $ i \in I $.
	If for all $ U \in \YY_0 $ the object $ \fupperstar(U) $ is coherent, the geometric morphism $ \flowerstar \colon \fromto{\XX}{\YY} $ is coherent.
\end{cor}

For the next result, we need the following lemma.

\begin{lem}\label{lem:sitecoherence}
	Let $ \fupperstar \colon \fromto{(Y,\tau_Y)}{(X,\tau_X)} $ be a morphism of $ \infty $-sites, and write $ \yo_Y \colon \fromto{Y}{\Sh_{\tau_Y}(Y)} $ for the sheafified Yoneda embedding.
	If the topology $ \tau_X $ is finitary, then
	\begin{equation*}
		\fupperstar \yo_Y \colon \fromto{Y}{\Sh_{\tau_X}(X)}
	\end{equation*}
	factors through $ \Sh_{\tau_X}(X)^{\coh} \subset \Sh_{\tau_X}(X) $.
\end{lem}

\begin{proof}
	We have a commutative square 
	\begin{equation*}
		\begin{tikzcd}
			Y \arrow[r, "\pupperstar"] \arrow[d, "\yo_{\!Y}" left] & X \arrow[d, "\yo_{\!\!X}" right] \\
			\Sh_{\tau_Y}(Y) \arrow[r, "\pupperstar" below] & \Sh_{\tau_X}(X)
		\end{tikzcd}
	\end{equation*}
	where the vertical functors are sheafified Yoneda embeddings.
	The claim now follows from the fact that $ \yo_X \colon \fromto{X}{\Sh_{\tau_X}(X)} $ factors through $ \Sh_{\tau_X}(X)^{\coh} $, since the topology $ \tau_X $ is finitary (\Cref{prop:SAG.A.3.1.3}). 
\end{proof}

\begin{cor}\label{cor:morsitescoherent}
	Let $ \fupperstar \colon \fromto{(Y,\tau_Y)}{(X,\tau_X)} $ be a morphism of finitary $ \infty $-sites.
	Then the geometric morphism 
	\begin{equation*}
		\flowerstar \colon \fromto{\Sh_{\tau_X}(X)}{\Sh_{\tau_Y}(Y)}
	\end{equation*}
	is coherent.
\end{cor}

\begin{proof}
	By \Cref{prop:SAG.A.3.1.3}, both $ \Sh_{\tau_X}(X) $ and $ \Sh_{\tau_Y}(Y) $ are locally coherent.
	The image $ \yo_Y(Y) $ of $ Y $ under the sheafified Yoneda embedding generates $ \Sh_{\tau_Y}(Y) $ under colimits, so by \Cref{cor:checkcohongen} it suffices to check that $ \fupperstar $ carries objects in $ \yo_Y(Y) $ to coherent objects of $ \XX $; this the content of \Cref{lem:sitecoherence}.
\end{proof}

\begin{ntn}
	Write $ \Top_{\infty}^{1,\coh} \subset \Topcoh $ for the full subcategory spanned by the $ 1 $-localic coherent $ \infty $-topoi.
\end{ntn}

\noindent \Cref{cor:morsitescoherent} and the definitions immediately imply the following:

\begin{prp}\label{prop:coherent1localic}
	The equivalence of $ \infty $-categories $ \tau_{\leq 0} \colon \equivto{\Top_{\infty}^{1}}{\Top} $ (\Cref{cnstr:localictopoi}) restricts to an equivalence
	\begin{equation*}
		\tau_{\leq 0} \colon \equivto{\Top_{\infty}^{1,\coh}}{\Topcohord}
	\end{equation*}
\end{prp}


\begin{cor}\label{cor:coherent1localicmors}
	The following are equivalent for a geometric morphism $ \flowerstar \colon \fromto{\XX}{\YY} $ between $ 1 $-localic coherent $ \infty $-topoi:
	\begin{enumerate}[{\upshape (\ref*{cor:coherent1localicmors}.1)}]
		\item The geometric morphism $ \flowerstar \colon \fromto{\XX}{\YY} $ is coherent.

		\item The pullback functor $ \fupperstar \colon \fromto{\YY}{\XX} $ carries $ 0 $-truncated $ 1 $-coherent objects of $ \YY $ to $ 1 $-coherent objects of $ \XX $.
	\end{enumerate}
\end{cor}

\begin{rmk}
	If $ n \geq 2 $, there doesn't already exist a notion of `coherent $ n $-topos' (other than saying that the corresponding $ n $-localic $ \infty $-topos is coherent).
	However, if one declares that an $ n $-topos $ \XX $ is `coherent' if $ \XX $ is `$ (n+1) $-coherent', then \Cref{cor:morsitescoherent} allows one to immediately deduce variants of \Cref{prop:coherent1localic,cor:coherent1localicmors} for coherent $ n $-topoi.
	Sections 5.4 through 5.6 of the newest version of \cite{exodromy} address this more general point.
\end{rmk}


\subsection{The \texorpdfstring{$\infty$}{∞}-pretopos associated to an ordinary pretopos}

In this subsection we exploit the equivalence of \Cref{prop:coherent1localic} to show how to associate a bounded $ \infty $-pretopos to an essentially small pretopos.
Lurie briefly touches upon this point (without details) in \cite{Lurie:CatLogic30}.

\begin{nul}
	If $ \XX $ is a bounded coherent $ \infty $-topos, then the associated ordinary topos $ \tau_{\leq 0} \XX $ is coherent.
	Moreover, if $ \flowerstar \colon \fromto{\XX}{\YY} $ is a coherent geometric morphism of bounded coherent $ \infty $-topoi, then the induced geometric morphsim $ \flowerstar \colon \fromto{\tau_{\leq 0} \XX}{\tau_{\leq 0} \YY} $ is a coherent geometric morphism of ordinary topoi.
	Hence the adjunction $ \Top_{\infty} \rightleftarrows \Top $
	restricts to an adjunction
	\begin{equation}\label{eq:cohadjunction}
		\begin{tikzcd}[column sep=2.5em]
			\Topbc \arrow[r, shift left, "\tau_{\leq 0}"] & \Topcohord \period \arrow[l, hooked', shift left]
		\end{tikzcd}	
	\end{equation}
\end{nul}

\begin{nul}
	Transporting the adjunction \eqref{eq:cohadjunction} across the equivalences 
	\begin{equation*}
		 (-)^{\coh} \colon \equivto{\Topcohord}{\preTopord^{\op}} \andeq (-)_{<\infty}^{\coh} \colon \equivto{\Topbc}{\preTop^{\bdd,\op}}
	\end{equation*}
	of \Cref{thm:classcoherent1topoi,thm:SAG.A.7.5.3} we see that the functor $ \tau_{\leq 0} \colon \fromto{\preTop^{\bdd}}{\preTopord} $ admits a fully faithful \textit{right} adjoint
	\begin{equation*}
		(-)^{+} \colon \incto{\preTopord}{\preTop^{\bdd}} 
	\end{equation*}
	given by $ X^{+} \coloneq \Sheff{X}_{<\infty}^{\coh} $.
\end{nul}

\begin{exm}
	The bounded $ \infty $-pretopos $ \Fin^{+} $ associated to the pretopos $ \Fin $ of finite sets is the $ \infty $-pretopos $ \Space_{\pi} $ of $ \pi $-finite spaces.
\end{exm}


\subsection{Examples from algebraic geometry}

We conclude with a few examples from algebraic geometry that \Cref{cor:morsitescoherent} puts on the same footing.

\begin{exm} 
	For a spectral topological space\footnote{A topological space $ S $ is \defn{spectral} if and only if $ S $ is homeomorphic to the underlying topological space of a quasicompact quasiseparated scheme.} $ S $, write $ \Open^{\qc}(S) \subset \Open(S) $ for the locale  of quasicompact opens in $ S $.
	Since the quasicompact opens of $ S $ form a basis for the topology on $ S $ that is closed under finite intersections, the $ \infty $-topos $ \Sh(\Open^{\qc}(S)) $ is $ 0 $-localic.
	Applying \cite[Proposition B.6.4]{Ultracategories} we see that the inclusion $ \Open^{\qc}(S) \subset \Open(S) $ induces an equivalence of $ 0 $-localic $ \infty $-topoi
	\begin{equation*}
		\Sh(S) \equivalent \Sh(\Open^{\qc}(S)) \period
	\end{equation*}
	The Grothendieck topology on \smash{$ \Open^{\qc}(S) $} is finitary, so the $ \infty $-topos $ \Sh(S) $ of sheaves on $ S $ is a coherent $ \infty $-topos.
	(Cf. \SAG{Lemma}{2.3.4.1}).

	If $ f \colon \fromto{S}{T} $ is a quasicompact continuous map of spectral topological spaces, the inverse image map $ f^{-1} \colon \fromto{\Open(T)}{\Open(S)} $ restricts to a map
	\begin{equation*}
		f^{-1} \colon \fromto{\Open^{\qc}(T)}{\Open^{\qc}(S)} \period
	\end{equation*}
	\Cref{cor:morsitescoherent} sho\-ws that the induced geometric morphism $ \flowerstar \colon \fromto{\Sh(S)}{\Sh(T)} $ is coherent.
	Since spectral topological spaces are sober, a continuous map $ f \colon \fromto{S}{T} $ of spectral topological spaces induces a coherent geometric morphism on the level of $ \infty $-topoi if and only if $ f $ is quasicompact.
\end{exm}

\begin{nul}
	If $ \XX $ is a coherent $ \infty $-topos, then the underlying topological space of $ \XX $ is spectral \cite[Chapter II, \S\S3.3--3.4]{MR861951}.
\end{nul}

Combining the fact that the Zariski, Nisnevich\footnote{For background on the Nisnevich topology, see \cites[\SAGsec{3.7}]{SAG}{Hoyois:Nisnevichnotes}{Hoyois:Nisnevichagree}{MR1045853}.}, étale, and proétale\footnote{For background on the proétale topology, see \cites[Tags \href{http://stacks.math.columbia.edu/tag/0988}{0988} \& \href{http://stacks.math.columbia.edu/tag/099R}{099R}]{stacksproject}{BhattScholzeProEtale}.} topoi of a scheme all have the same underlying topological space with the fact that if a scheme $ X $ is quasicompact and quasiseparated, then the topoi of sheaves on $ X $ in each of these topologies is coherent \cites[\SAGthm{Proposition}{2.3.4.2} \& \SAGthm{Remark}{3.7.4.2}]{SAG}[Appendix A]{MotivicNorms:BachmannHoyois}[Example 7.1.7]{Ultracategories}, we deduce the following: 

\begin{prp}\label{prop:coherentschemetopos}
	The following are equivalent for a scheme $ X $:
	\begin{enumerate}[{\upshape (\ref*{prop:coherentschemetopos}.1)}]
		\item The scheme $ X $ is quasicompact and quasiseparated.

		\item The Zariski $ \infty $-topos $ X_{\zar} $ of $ X $ is a coherent $ \infty $-topos.

		\item The Nisnevich $ \infty $-topos $ X_{\Nis} $ of $ X $ is a coherent $ \infty $-topos.

		\item The étale $ \infty $-topos $ X_{\et} $ of $ X $ is a coherent $ \infty $-topos.

		\item The proétale $ \infty $-topos $ X_{\proet} $ of $ X $ is a coherent $ \infty $-topos.
	\end{enumerate}
\end{prp}


\begin{exm}[{\cite[Example 10.4.13]{exodromy}}]
	Let $ X $ be a quasicompact quasiseparated scheme.
	Then the bounded $ \infty $-pretopos of truncated coherent objects of the coherent $ \infty $-topos $ X_{\et} $ is the $ \infty $-category of constructible étale sheaves of spaces on $ X $.
\end{exm}

\begin{exm}
	Let $ f \colon \fromto{X}{Y} $ be a morphism of quasicompact quasiseparated schemes and let $ \tau \in \{\zar,\Nis,\et,\proet\} $.
	Then the induced geometric morphism $ \flowerstar \colon \fromto{X_{\tau}}{Y_{\tau}} $ on $ \infty $-topoi of $ \tau $-sheaves is a coherent geometric morphism of coherent $ \infty $-topoi. 
\end{exm}

\begin{exm}
	Let $ X $ be a quasicompact quasiseparated scheme.
	Then the natural geometric morphisms
	\begin{equation*}
		\fromto{X_{\proet}}{X_{\et}} \comma \qquad \fromto{X_{\et}}{X_{\Nis}} \comma \andeq \fromto{X_{\Nis}}{X_{\zar}}
	\end{equation*}
	are all coherent geometric morphisms of coherent $ \infty $-topoi.
\end{exm}


\DeclareFieldFormat{labelnumberwidth}{#1}
\printbibliography[keyword=alph]
\addcontentsline{toc}{section}{References} 
\DeclareFieldFormat{labelnumberwidth}{{#1\adddot\midsentence}}
\printbibliography[heading=none, notkeyword=alph]

\end{document}